\documentclass[reqno]{amsart}

\usepackage{amsmath}
\usepackage{amsthm}
\usepackage{amscd}
\usepackage{bbm}
\usepackage{enumerate}
\usepackage{amssymb}
\usepackage{palatino}
\usepackage{amscd}
\usepackage{verbatim}
\usepackage{mathrsfs}
\usepackage[margin=1.5in]{geometry}



\title{Proof of a conjecture of Stanley-Zanello}
\author{Levent Alpoge}\thanks{Email: alpoge@college.harvard.edu.}\email{alpoge@college.harvard.edu}
\address{Quincy House, Harvard College, Cambridge, MA 02138.}

\begin{document}

\begin{abstract}
We prove that the number of partitions of an integer into at most $b$ distinct parts of size at most $n$ forms a unimodal sequence for $n$ sufficiently large with respect to $b$. This resolves a recent conjecture of Stanley and Zanello.
\end{abstract}

\maketitle

\newtheoremstyle{dotless}{}{}{\itshape}{}{\bfseries}{}{ }{}

\newtheorem{thm}{Theorem}
\newtheorem{lem}[thm]{Lemma}
\newtheorem{remark}[thm]{Remark}
\newtheorem{cor}[thm]{Corollary}
\newtheorem{defn}[thm]{Definition}
\newtheorem{prop}[thm]{Proposition}
\newtheorem{conj}[thm]{Conjecture}
\newtheorem{claim}[thm]{Claim}
\newtheorem{exer}[thm]{Exercise}
\newtheorem{fact}[thm]{Fact}

\theoremstyle{dotless}

\newtheorem{thmnodot}[thm]{Theorem}
\newtheorem{lemnodot}[thm]{Lemma}
\newtheorem{cornodot}[thm]{Corollary}

\newcommand{\image}{\mathop{\text{image}}}
\newcommand{\End}{\mathop{\text{End}}}
\newcommand{\Hom}{\mathop{\text{Hom}}}
\newcommand{\Sum}{\displaystyle\sum\limits}
\newcommand{\Prod}{\displaystyle\prod\limits}
\newcommand{\Tr}{\mathop{\mathrm{Tr}}}
\renewcommand{\Re}{\operatorname{\mathfrak{Re}}}
\renewcommand{\Im}{\operatorname{\mathfrak{Im}}}
\newcommand{\im}{\mathrm{im}\,}
\newcommand{\inner}[1]{\langle #1 \rangle}
\newcommand{\pair}[2]{\langle #1, #2\rangle}
\newcommand{\ppair}[2]{\langle\langle #1, #2\rangle\rangle}
\newcommand{\Pair}[2]{\left[#1, #2\right]}
\newcommand{\Char}{\mathop{\mathrm{char}}}
\newcommand{\rank}{\mathop{\mathrm{rank}}}
\newcommand{\sgn}[1]{\mathop{\mathrm{sgn}}(#1)}
\newcommand{\leg}[2]{\left(\frac{#1}{#2}\right)}
\newcommand{\Sym}{\mathrm{Sym}}
\newcommand{\hmat}[2]{\left(\begin{array}{cc} #1 & #2\\ -\bar{#2} & \bar{#1}\end{array}\right)}
\newcommand{\HMat}[2]{\left(\begin{array}{cc} #1 & #2\\ -\overline{#2} & \overline{#1}\end{array}\right)}
\newcommand{\Sin}[1]{\sin{\left(#1\right)}}
\newcommand{\Cos}[1]{\cos{\left(#1\right)}}
\newcommand{\comm}[2]{\left[#1, #2\right]}
\newcommand{\Isom}{\mathop{\mathrm{Isom}}}
\newcommand{\Map}{\mathop{\mathrm{Map}}}
\newcommand{\Bij}{\mathop{\mathrm{Bij}}}
\newcommand{\Z}{\mathbb{Z}}
\newcommand{\R}{\mathbb{R}}
\newcommand{\Q}{\mathbb{Q}}
\newcommand{\C}{\mathbb{C}}
\newcommand{\Nm}{\mathrm{Nm}}
\newcommand{\RI}[1]{\mathcal{O}_{#1}}
\newcommand{\F}{\mathbb{F}}
\renewcommand{\Pr}{\displaystyle\mathop{\mathrm{Pr}}\limits}
\newcommand{\E}{\mathbb{E}}
\newcommand{\coker}{\mathop{\mathrm{coker}}}
\newcommand{\id}{\mathop{\mathrm{id}}}
\newcommand{\Oplus}{\displaystyle\bigoplus\limits}
\renewcommand{\Cap}{\displaystyle\bigcap\limits}
\renewcommand{\Cup}{\displaystyle\bigcup\limits}
\newcommand{\Bil}{\mathop{\mathrm{Bil}}}
\newcommand{\N}{\mathbb{N}}
\newcommand{\Aut}{\mathop{\mathrm{Aut}}}
\newcommand{\ord}{\mathop{\mathrm{ord}}}
\newcommand{\ch}{\mathop{\mathrm{char}}}
\newcommand{\minpoly}{\mathop{\mathrm{minpoly}}}
\newcommand{\Spec}{\mathop{\mathrm{Spec}}}
\newcommand{\Gal}{\mathop{\mathrm{Gal}}}
\newcommand{\Ad}{\mathop{\mathrm{Ad}}}
\newcommand{\Stab}{\mathop{\mathrm{Stab}}}
\newcommand{\Norm}{\mathop{\mathrm{Norm}}}
\newcommand{\Orb}{\mathop{\mathrm{Orb}}}
\newcommand{\pfrak}{\mathfrak{p}}
\newcommand{\qfrak}{\mathfrak{q}}
\newcommand{\mfrak}{\mathfrak{m}}
\newcommand{\Frac}{\mathop{\mathrm{Frac}}}
\newcommand{\Loc}{\mathop{\mathrm{Loc}}}
\newcommand{\Sat}{\mathop{\mathrm{Sat}}}
\newcommand{\inj}{\hookrightarrow}
\newcommand{\surj}{\twoheadrightarrow}
\newcommand{\bij}{\leftrightarrow}
\newcommand{\Ind}{\mathrm{Ind}}
\newcommand{\Supp}{\mathop{\mathrm{Supp}}}
\newcommand{\Ass}{\mathop{\mathrm{Ass}}}
\newcommand{\Ann}{\mathop{\mathrm{Ann}}}
\newcommand{\Krulldim}{\dim_{\mathrm{Kr}}}
\newcommand{\Avg}{\mathop{\mathrm{Avg}}}
\newcommand{\innerhom}{\underline{\Hom}}
\newcommand{\triv}{\mathop{\mathrm{triv}}}
\newcommand{\Res}{\mathrm{Res}}
\newcommand{\eval}{\mathop{\mathrm{eval}}}
\newcommand{\MC}{\mathop{\mathrm{MC}}}
\newcommand{\Fun}{\mathop{\mathrm{Fun}}}
\newcommand{\InvFun}{\mathop{\mathrm{InvFun}}}
\renewcommand{\ch}{\mathop{\mathrm{ch}}}
\newcommand{\irrep}{\mathop{\mathrm{Irr}}}
\newcommand{\len}{\mathop{\mathrm{len}}}
\newcommand{\SL}{\mathrm{SL}}
\newcommand{\GL}{\mathrm{GL}}
\newcommand{\PSL}{\mathrm{SL}}
\newcommand{\actson}{\curvearrowright}
\renewcommand{\H}{\mathbb{H}}
\newcommand{\mat}[4]{\left(\begin{array}{cc} #1 & #2\\ #3 & #4\end{array}\right)}
\newcommand{\interior}{\mathop{\mathrm{int}}}
\newcommand{\floor}[1]{\lfloor #1\rfloor}
\newcommand{\iso}{\cong}
\newcommand{\eps}{\epsilon}
\newcommand{\disc}{\mathrm{disc}}
\newcommand{\Frob}{\mathrm{Frob}}
\newcommand{\charpoly}{\mathrm{charpoly}}
\newcommand{\afrak}{\mathfrak{a}}
\newcommand{\cfrak}{\mathfrak{c}}
\newcommand{\codim}{\mathrm{codim}}
\newcommand{\ffrak}{\mathfrak{f}}
\newcommand{\Pfrak}{\mathfrak{P}}
\newcommand{\homcont}{\hom_{\mathrm{cont}}}
\newcommand{\vol}{\mathrm{vol}}
\newcommand{\ofrak}{\mathfrak{o}}
\newcommand{\A}{\mathbb{A}}
\newcommand{\I}{\mathbb{I}}
\newcommand{\invlim}{\varprojlim}
\newcommand{\dirlim}{\varinjlim}
\renewcommand{\ch}{\mathrm{char}}
\newcommand{\artin}[2]{\left(\frac{#1}{#2}\right)}
\newcommand{\Qfrak}{\mathfrak{Q}}
\newcommand{\ur}[1]{#1^{\mathrm{ur}}}
\newcommand{\absnm}{\mathcal{N}}
\newcommand{\ab}[1]{#1^{\mathrm{ab}}}
\newcommand{\G}{\mathbb{G}}
\newcommand{\dfrak}{\mathfrak{d}}
\newcommand{\Bfrak}{\mathfrak{B}}
\renewcommand{\sgn}{\mathrm{sgn}}
\newcommand{\disjcup}{\bigsqcup}
\newcommand{\zfrak}{\mathfrak{z}}
\renewcommand{\Tr}{\mathrm{Tr}}
\newcommand{\reg}{\mathrm{reg}}
\newcommand{\subgrp}{\leq}
\newcommand{\normal}{\vartriangleleft}
\newcommand{\Dfrak}{\mathfrak{D}}
\newcommand{\nvert}{\nmid}
\newcommand{\K}{\mathbb{K}}
\newcommand{\pt}{\mathrm{pt}}
\newcommand{\RP}{\mathbb{RP}}
\newcommand{\CP}{\mathbb{CP}}
\newcommand{\rk}{\mathop{\mathrm{rk}}}
\newcommand{\redH}{\tilde{H}}
\renewcommand{\H}{\tilde{H}}
\newcommand{\Cyl}{\mathrm{Cyl}}
\newcommand{\T}{\mathbb{T}}
\newcommand{\Ab}{\mathrm{Ab}}
\newcommand{\Vect}{\mathrm{Vect}}
\newcommand{\Top}{\mathrm{Top}}
\newcommand{\Nat}{\mathrm{Nat}}
\newcommand{\inc}{\mathrm{inc}}
\newcommand{\Tor}{\mathrm{Tor}}
\newcommand{\Ext}{\mathrm{Ext}}
\newcommand{\fungrpd}{\pi_{\leq 1}}
\newcommand{\slot}{\mbox{---}}
\newcommand{\funct}{\mathcal}
\newcommand{\Funct}{\mathcal{F}}
\newcommand{\Gunct}{\mathcal{G}}
\newcommand{\FunCat}{\mathrm{Funct}}
\newcommand{\Rep}{\mathrm{Rep}}
\newcommand{\Specm}{\mathrm{Specm}}
\newcommand{\ev}{\mathrm{ev}}
\newcommand{\frpt}[1]{\{#1\}}
\newcommand{\h}{\mathscr{H}}
\newcommand{\poly}{\mathrm{poly}}
\newcommand{\Partial}[1]{\frac{\partial}{\partial #1}}
\newcommand{\Cont}{\mathrm{Cont}}
\renewcommand{\o}{\ofrak}
\newcommand{\bfrak}{\mathfrak{b}}
\newcommand{\Cl}{\mathrm{Cl}}
\newcommand{\ceil}[1]{\lceil #1\rceil}
\newcommand{\hfrak}{\mathfrak{h}}

\let\uglyphi\phi
\let\phi\varphi

\tableofcontents

\section{Introduction}

%Here we prove that the number of partitions of an integer into at most $b$ distinct parts of size at most $n$ forms a unimodal sequence for $n$ %sufficiently large with respect to $b$. This resolves a conjecture of Stanley and Zanello (cf.\ Conjecture 3.9 of \cite{stanleyzanello}).

The unimodality (that is, weak increase followed by weak decrease of the coefficients) of the Gaussian binomial coefficients ${n\choose k}_q$ is a classical result with many proofs, the first given by Sylvester in his work on invariant theory. By interpreting the coefficient of $q^\ell$ in ${n\choose k}_q$ as the number of partitions of $\ell$ with Young diagram fitting inside an $(n-k)\times k$ box, one is naturally led to ask similar questions on unimodality upon changing the ambient shape from a box to something more exotic. Stanton \cite{stanton} was the first to study such questions, obtaining various infinite families of partitions leading to nonunimodal sequences, as well as unimodality results for partitions with at most three parts. Stanley and Zanello \cite{stanleyzanello} then considered the question of counting partitions \emph{with distinct parts} fitting inside these shapes, allowing them to progress on analogous problems. In their paper they stated various conjectures in this direction. In this work we prove one of their conjectures, having to do with partitions with distinct parts fitting inside ``truncated staircases''.

\section{Preliminaries}

By $A\ll_\theta B$ we mean $|A|\leq C|B|$ for some positive constant $C$ potentially depending on $\theta$. By $A\asymp B$ we mean $A\ll B$ and $B\ll A$. We will use the notation $e(z):=\exp(2\pi i z)$, and for us the Fourier transform is \begin{align}\hat{f}(\xi):=\int_\R f(x)e(-x\xi)dx.\end{align} For us $O(B)$ will denote a quantity $\ll B$, and $f*g(x) = \int_\R f(t) g(x-t) dt$ will denote the convolution of $f$ and $g$.

Recall that a sequence $a_i$ is \emph{unimodal} if there is some $k$ for which $a_1\leq\cdots\leq a_k\geq\cdots$. We will call a polynomial unimodal if its coefficients (in increasing order of degree) form a unimodal sequence.

The Gaussian binomial coefficient is defined as \begin{align}{n\choose b}_q:=\frac{\prod_{k=0}^{b-1} (1-q^{n-k})}{\prod_{k=1}^b (1-q^k)}.\end{align} The coefficient of $q^\ell$ in ${n\choose b}_q$ is the number of partitions of $\ell$ into at most $b$ parts each of size at most $n-b$. It is a theorem of Sylvester (essentially proved in 1878 after it was conjectured by Cayley some 20 years before) that the coefficients of $q$ in ${n\choose b}_q$ form a unimodal sequence --- that is, the polynomial ${n\choose b}_q$ is \emph{unimodal}.

\section{Main result}

In their paper about unimodality of partitions with distinct parts with Young diagrams fitting inside certain shapes, Stanley and Zanello \cite{stanleyzanello} conjecture that the number of partitions into at most $b$ distinct parts each of size at most $n$ forms a unimodal sequence for $n$ large with respect to $b$. We prove this conjecture here.

To set notation, let $\lambda_{n,b}:=(n,\ldots,n-b+1)$. Let $c_{n,b}(\ell)$ be the number of partitions of $\ell$ into distinct parts with Young diagrams fitting inside $\lambda_{n,b}$ (that is, those partitions $\lambda_1 > \cdots > \lambda_r$ with $r\leq b$, $\lambda_1\leq n$, and $\sum \lambda_i = \ell$). Let $F_{\lambda_{n,b}}(q):=\sum_{\ell\geq 0} c_{n,b}(\ell) q^\ell$, the \emph{rank-generating function} of the $c_{n,b}(\ell)$.

\begin{thmnodot}[Cf.\ Conjecture 3.9 of \cite{stanleyzanello}.]\label{maintheorem}
For $n\gg_b 1$, $F_{\lambda_{n,b}}(q)$ is unimodal.
\end{thmnodot}

The hypothesis $n\gg_b 1$ is in fact necessary, as Stanley and Zanello note --- for $n=19, b=6$ the claim fails. It is a theorem of Dynkin that for $n\leq b$ (that is, for ``nontruncated staircases'') the polynomial $F_{\lambda_{n,b}}$ is unimodal. We will prove Theorem \ref{maintheorem} by employing analytic techniques.

The following is an outline of the arguments we will employ. In the first place, for coefficients near the ``edges'' (i.e., near the constant term or the top term), we will employ elementary methods to show the desired inequalities. Next, for the remaining coefficients (all of which are a distance of order $n$ from the edges) we write the coefficient $c_{n,b}(\ell)$ as a smooth function of a real variable $\ell$, say $f(\ell)$. We show $f$ is log-concave once $n\gg_b 1$ by calculating $(-\log{f})''(x)$ to leading order in $n$, getting $\gamma_b(x)\cdot n^{-2} + O_b(n^{-3})$, where $\gamma_b > 0$ is a strictly positive function of $x$ depending only on $b$. By our work at the edges we may restrict to $x$ to a subinterval, whence by compactness this is positive for $n\gg_b 1$, finishing the proof.

\section{Proof of Theorem \ref{maintheorem}}

Stanley and Zanello handle the cases $b\leq 4$ in their paper, so we are free to assume $b\geq 5$.

\subsection{Handling the tails.}

First we prove that \begin{align}c_{n,b}(0)\leq\cdots\leq c_{n,b}(n)\end{align} and \begin{align}c_{n,b}((b-1)n)\geq\cdots\geq c_{n,b}\left(bn - \frac{b(b-1)}{2}\right).\end{align} For the first claim, suppose $\ell < n$. Let \begin{align}\lambda =: \lambda_1 > \cdots > \lambda_r\end{align} be a partition of $\ell$ into $r\leq b$ parts. Of course $\lambda_1\leq \ell < n$ so $\lambda$ automatically fits into $\lambda_{n,b}$. But then so does \begin{align}\mu:=(\lambda_1 + 1) > \lambda_2 > \cdots > \lambda_r.\end{align} This procedure is of course injective, whence the claimed chain of inequalities.

For the second claim, if $\ell > (b-1)n$ and \begin{align}\lambda =: \lambda_1 > \cdots > \lambda_r\end{align} is a partition of $\ell$ into $r\leq b$ parts with $\lambda_1\leq n$, then $r = b$ since \begin{align}(b-1)n < \ell = \sum \lambda_i\leq r\lambda_1\leq rn.\end{align} Hence we may form the partition \begin{align}\mu:=\lambda_1 > \cdots > \lambda_{b-1} > (\lambda_b - 1)\end{align} (omitting the last term if it is zero) of $\ell - 1$. Again this is evidently injective, whence the second claim follows.

So in what follows we'll take $\frac{1728}{1729}\leq \frac{\ell}{n}\leq b-\frac{1728}{1729}$.\footnote{The constant $\frac{1728}{1729}$ does not matter --- any constant sufficiently close to $1$ will do.} In this range we will prove that, in fact, $c_{n,b}(\ell)$ is logarithmically concave --- that is, \begin{align}c_{n,b}(\ell)^2\geq c_{n,b}(\ell-1)\cdot c_{n,b}(\ell+1).\end{align} (Note that this immediately implies unimodality.)

To do this we will calculate $c_{n,b}(\ell)$ to leading order in $n$ and prove the inequality at the level of leading terms. Taking $n$ large will yield the inequality for $c_{n,b}(\ell)$ proper.

\subsection{Handling the bulk.}

Before beginning it is worth remarking that, heuristically, for $\ell\asymp_b n$ we expect the number of partitions of $\ell$ into at most $b$ distinct parts of size at most $n$ to be basically governed by the distribution of a sum of $b$ uniform random variables on $[0,n]$. We will see that, to leading order, this is exactly the case.

Now let us begin the calculation. Again, we recall that now $\ell\asymp_b n$.

Let $\alpha:=\frac{1}{n}$.\footnote{This is the optimal order of magnitude for $\alpha$ for our purposes. Its choice should be thought of as adhering to the method of steepest descent.}

Of course \begin{align}c_{n,b}(\ell) = \int_{-\frac{1}{2}}^{\frac{1}{2}} F_{\lambda_{n,b}}(e(\theta + i\alpha)) e(-\ell(\theta + i\alpha))d\theta.\end{align}

Now, a partition of $\ell$ into distinct parts $n\geq \lambda_1 > \cdots > \lambda_r$ with $r\leq b$ is precisely equivalent to the partition $(\lambda_1 - r)\geq (\lambda_2 - r + 1)\geq\cdots\geq (\lambda_{r-1} - 2)\geq (\lambda_r - 1)$ (omitting zeroes if there are any) of $\ell - \frac{r(r+1)}{2}$ of size at most $n - r$. That is, the partitions of $\ell$ fitting inside $\lambda_{n,b}$ with \emph{exactly} $r$ parts are equinumerous with those partitions of $\ell - \frac{r(r+1)}{2}$ fitting inside an $(n-r)\times r$ box. Hence the contribution of those partitions with exactly $r$ parts to $F_{\lambda_{n,b}}(q)$ is $q^{\frac{r(r+1)}{2}}{n\choose r}_q$. That is to say, \begin{align}F_{\lambda_{n,b}}(q) = \sum_{a=0}^b q^{\frac{a(a+1)}{2}}{n\choose a}_q.\end{align}

Hence we see that \begin{align}c_{n,b}(\ell) = \sum_{a=1}^b \int_{-\frac{1}{2}}^{\frac{1}{2}} d\theta\ \frac{\prod_{k=0}^{a-1} (1 - e(\theta + i\alpha)^{n-k})}{\prod_{k=1}^a (1 - e(\theta + i\alpha)^k)}\cdot e\left(-\left(\ell - \frac{a(a+1)}{2}\right)(\theta + i\alpha)\right).\end{align} We have dropped the $a=0$ term for convenience (as we well may, since $\ell > 2$).

To prove the inequality $c_{n,b}(\ell)^2\geq c_{n,b}(\ell-1)\cdot c_{n,b}(\ell+1)$, observe that the integral formula for $c_{n,b}(\ell)$ extends to a function (let us call it $f(\ell)$) of a \emph{real} variable $\ell$.\footnote{Note that $f$ is also real-valued, since complex conjugation amounts to $\theta\mapsto -\theta$ in the integral.} That is, \begin{align}f(\ell):=\sum_{a=1}^b \int_{-\frac{1}{2}}^{\frac{1}{2}} d\theta\ \frac{\prod_{k=0}^{a-1} (1 - e(\theta + i\alpha)^{n-k})}{\prod_{k=1}^a (1 - e(\theta + i\alpha)^k)}\cdot e\left(-\left(\ell - \frac{a(a+1)}{2}\right)(\theta + i\alpha)\right).\end{align} Of course $f$ is smooth, and to prove the claimed inequality it suffices to prove that $f$ is logarithmically concave --- that is, \begin{align}(f')^2\geq f\cdot f''.\end{align} Written another way (in the notation of Odlyzko-Richmond \cite{odlyzkorichmond}), letting \begin{align}J_m := f^{(m)}(\ell) = (-2\pi i)^m\sum_{a=1}^b \int_{-\frac{1}{2}}^{\frac{1}{2}} d\theta\ (\theta + i\alpha)^m \frac{\prod_{k=0}^{a-1} (1 - e(\theta + i\alpha)^{n-k})}{\prod_{k=1}^a (1 - e(\theta + i\alpha)^k)} e\left(-\left(\ell - \frac{a(a+1)}{2}\right)(\theta + i\alpha)\right),\end{align} we will show that $J_1^2 > J_0 J_2$.

To do this we will calculate $J_m$ ($0\leq m\leq 2$) to leading order.

In this vein, note that, for $0\leq k\leq b$,
\begin{align}
1 - e(\theta + i\alpha)^{n-k} &= 1 - e(n\theta + i) + O_b\left(|\theta| + \frac{1}{n}\right)\\&= \left(1 - e(n\theta + i)\right)\left(1 + O_b\left(|\theta| + \frac{1}{n}\right)\right),
\end{align}
since $1 - e(n\theta + i)\gg 1$. Similarly,
\begin{align}
1 - e(\theta + i\alpha)^k &= k(-2\pi i\theta + 2\pi\alpha) + O_b\left(\left(|\theta| + \frac{1}{n}\right)^2\right) \\&= k(-2\pi i\theta + 2\pi\alpha)\left(1 + O_b\left(|\theta| + \frac{1}{n}\right)\right).
\end{align}
Thus
\begin{align}
J_m &= (-2\pi i)^m\sum_{a=1}^b \frac{1}{a!}\int_{-\frac{1}{2}}^{\frac{1}{2}} d\theta\ (\theta + i\alpha)^m \left(\frac{1 - e(n\theta + i)}{-2\pi i\theta + 2\pi\alpha}\right)^a e\left(-\left(\ell - \frac{a(a+1)}{2}\right)(\theta + i\alpha)\right)\nonumber\\&\quad\quad + O_b\left(\sum_{a=1}^b \int_{-\frac{1}{2}}^{\frac{1}{2}} |\theta + i\alpha|^{m+1-a} d\theta\right).
\end{align}

Via $|\theta + i\alpha|\asymp |\theta|$ if $|\theta| > \frac{1}{n}$ and $\asymp \frac{1}{n}$ otherwise, by splitting the integral into two integrals over $\left[0,\frac{1}{n}\right]$ and $\left[\frac{1}{n},\frac{1}{2}\right]$, respectively, we obtain the bound \begin{align}\sum_{a=1}^b \int_{-\frac{1}{2}}^{\frac{1}{2}} |\theta + i\alpha|^{m+1-a} d\theta&\ll_b \sum_{a=1}^b \left(n^{a-m-2} + \int_{\frac{1}{n}}^{\frac{1}{2}} \theta^{m+1-a} d\theta\right)\\&\ll_b n^{b-m-2}\end{align} (here we use that $b-m-2 > 0$ since $b\geq 5$ and $m\leq 2$).

Hence we obtain that
\begin{align}
J_m &= (-2\pi i)^m\sum_{a=1}^b \frac{1}{a!}\int_{-\frac{1}{2}}^{\frac{1}{2}} d\theta\ (\theta + i\alpha)^m \left(\frac{1 - e(n\theta + i)}{-2\pi i\theta + 2\pi\alpha}\right)^a e\left(-\left(\ell - \frac{a(a+1)}{2}\right)(\theta + i\alpha)\right) \nonumber\\&\quad\quad+ O_b\left(n^{b-m-2}\right).
\end{align}

As a final step, we bound the terms with $a<b$ trivially (in exactly the same way) to obtain
\begin{align}
&(-2\pi i)^m\sum_{a=1}^{b-1} \frac{1}{a!}\int_{-\frac{1}{2}}^{\frac{1}{2}} d\theta\ (\theta + i\alpha)^m \left(\frac{1 - e(n\theta + i)}{-2\pi i\theta + 2\pi\alpha}\right)^a e\left(-\left(\ell - \frac{a(a+1)}{2}\right)(\theta + i\alpha)\right)\\&\quad\quad\ll_b \sum_{a=1}^{b-1} \int_{-\frac{1}{2}}^{\frac{1}{2}} d\theta\ |\theta + i\alpha|^{m-a}\\&\quad\quad\ll_b n^{b-m-2}.
\end{align}

Thus we find that
\begin{align}
J_m &= \frac{(-2\pi i)^m}{b!}\int_{-\frac{1}{2}}^{\frac{1}{2}} d\theta\ (\theta + i\alpha)^m \left(\frac{1 - e(n\theta + i)}{-2\pi i\theta + 2\pi\alpha}\right)^b e\left(-\left(\ell - \frac{b(b+1)}{2}\right)(\theta + i\alpha)\right) \nonumber\\&\quad\quad+ O_b\left(n^{b-m-2}\right).
\end{align}

Now we study the main term. Via $\theta\mapsto \theta/n$, we get
\begin{align}
J_m &= \frac{(-2\pi i)^m}{b!} n^{b-m-1}\int_{-\frac{n}{2}}^{\frac{n}{2}} d\theta\ (\theta + i)^m \left(\frac{1 - e(\theta + i)}{-2\pi i\theta + 2\pi}\right)^b e\left(-\left(\frac{\ell}{n} - \frac{b(b+1)}{2n}\right)(\theta + i)\right) \nonumber\\&\quad\quad+ O_b\left(n^{b-m-2}\right).
\end{align}

Note that (again bounding trivially) \begin{align}\int_{|\theta| > \frac{n}{2}} d\theta\ (\theta + i)^m \left(\frac{1 - e(\theta + i)}{-2\pi i\theta + 2\pi}\right)^b e\left(-\left(\frac{\ell}{n} - \frac{b(b+1)}{2n}\right)(\theta + i)\right)\ll_b n^{-(b-m-1)}.\end{align} Thus we may extend the integral to all of $\R$ and absorb the error into the existing error term. That is,
\begin{align}
J_m &= \frac{(-2\pi i)^m}{b!} n^{b-m-1}\int_\R d\theta\ (\theta + i)^m \left(\frac{1 - e(\theta + i)}{-2\pi i\theta + 2\pi}\right)^b e\left(-\left(\frac{\ell}{n} - \frac{b(b+1)}{2n}\right)(\theta + i)\right) \nonumber\\&\quad\quad+ O_b\left(n^{b-m-2}\right).
\end{align}

A contour shift from $\R + i$ to $\R$ ($\frac{1 - e(z)}{z}$ is entire) tells us that
\begin{align}
&\int_\R d\theta\ (\theta + i)^m \left(\frac{1 - e(\theta + i)}{-2\pi i\theta + 2\pi}\right)^b e\left(-\left(\frac{\ell}{n} - \frac{b(b+1)}{2n}\right)(\theta + i)\right) \\&= \int_\R d\theta\ \theta^m \left(\frac{1 - e(\theta)}{-2\pi i\theta}\right)^b e\left(-\left(\frac{\ell}{n} - \frac{b(b+1)}{2n}\right)\theta\right).
\end{align}

That is (via $\theta\mapsto -\theta$),
\begin{align}
J_m &= \frac{n^{b-m-1}}{b!}\int_\R d\theta\ (2\pi i\theta)^m \left(\frac{1 - e(-\theta)}{2\pi i\theta}\right)^b e\left(\left(\frac{\ell}{n} - \frac{b(b+1)}{2n}\right)\theta\right) \nonumber\\&\quad\quad+ O_b\left(n^{b-m-2}\right).
\end{align}

Immediately we recognize $\frac{1 - e(-\theta)}{2\pi i\theta}$ as the Fourier transform of $\chi_{[0,1]}$, the characteristic function of $[0,1]$. That is, writing $I:=\chi_{[0,1]}^{*b}$ for the $b$-fold convolution of this characteristic function with itself (this is the probability density function of the Irwin-Hall distribution --- i.e., that of the sum of $b$ independent uniform random variables on $[0,1]$), we find that\footnote{Here $[(I')^2 - I I'']\left(\frac{\ell}{n} - \frac{b(b+1)}{2n}\right)$ is the evaluation of the function $(I')^2 - I I''$ at the point $\frac{\ell}{n} - \frac{b(b+1)}{2n}$, rather than the product of two terms, as the typesetting suggests. The same expression occurs below.}
\begin{align}
J_1^2 - J_0 J_2 &= \frac{n^{2b-4}}{(b!)^2} \left([(I')^2 - I I'']\left(\frac{\ell}{n} - \frac{b(b+1)}{2n}\right)\right) \nonumber\\&\quad\quad+ O_b\left(n^{2b-5}\right).
\end{align}

Since the error term is of lower order than the main term, it suffices to show that \begin{align}[(I')^2 - I I'']\left(\frac{\ell}{n} - \frac{b(b+1)}{2n}\right)\gg_b 1.\end{align} Since, for $n\gg_b 1$, \begin{align}\frac{1}{2}\leq \frac{1728}{1729} - O_b(n^{-1})\leq \frac{\ell}{n} - \frac{b(b+1)}{2n}\leq b - \frac{1728}{1729} + O_b(n^{-1})\leq b - \frac{1}{2},\end{align} it suffices to show that \begin{align}(I')^2 - I I''\gg_b 1\end{align} on $\left[\frac{1}{2}, b - \frac{1}{2}\right]$.

We will show that \begin{align}(-\log{I})''\gg_b 1\end{align} on this interval by a method found in \cite{henningssonastrom} (--- the argument here is exactly the same. We provide it only for completeness.).

\begin{lem}\label{convolution}
Let $f,g > 0$ be $C^2$ and such that \begin{align}(-\log{f})''&\geq \frac{1}{A},\\(-\log{g})''&\geq \frac{1}{B}.\end{align} Then \begin{align}(-\log(f*g))''\geq \frac{1}{A+B}.\end{align}
\end{lem}

Assuming this, since by inspection $(-\log{\chi_{[0,1]}^{*a}})''$ is nonzero on $(0,a)$ for $a=4,5,6,7$, by compactness \begin{align}(-\log{\chi_{[0,1]}^{*a}})''\gg_b 1\end{align} on $\left[\frac{1}{2\floor{\frac{b}{4}}}, a - \frac{1}{2\floor{\frac{b}{4}}}\right]$ for $a=4,5,6,7$. Hence, by the lemma, \begin{align}(-\log{I})''\gg_b 1\end{align} on $\left[\frac{1}{2}, b - \frac{1}{2}\right]$ since \begin{align}I = (\chi_{[0,1]}^{*4})^{*(\floor{\frac{b}{4}} - 1)} * \chi_{[0,1]}^{*(4 + (b\bmod{4}))}.\end{align} Therefore, taking $n\gg_b 1$, we see that \begin{align}J_1^2 > J_0 J_2\end{align} for all $\frac{1728}{1729} n\leq \ell\leq \left(b - \frac{1728}{1729}\right)n$, establishing log-concavity in this interval and hence unimodality for all $\ell$.

Thus it remains to prove the lemma.

\begin{proof}[Proof of Lemma \ref{convolution}.]
Let \begin{align}F(x):=f(x) e^{\frac{x^2}{2A}}\end{align} and \begin{align}G(x):=g(x) e^{\frac{x^2}{2B}}.\end{align} Note that $F$ and $G$ are, by hypothesis, log-concave. But
\begin{align}
e^{\frac{x^2}{2(A+B)}} (f*g)(x) &= \int_\R e^{\frac{x^2}{2(A+B)}} f(t) g(x-t) dt\\&= \int_\R \exp\left(-\frac{t^2}{2A} - \frac{(x-t)^2}{2B} + \frac{x^2}{2(A+B)}\right) F(t) G(x-t) dt.
\end{align}
Notice that the Hessian of the exponent is \begin{align}\left(\begin{array}{cc} \frac{1}{A} + \frac{1}{B} & \frac{1}{B}\\ \frac{1}{B} & \frac{A}{B(A+B)}\end{array}\right),\end{align} which is nonnegative-definite (it has zero determinant and positive trace). Hence the function \begin{align}h(x,t):=\exp\left(-\frac{t^2}{2A} - \frac{(x-t)^2}{2B} + \frac{x^2}{2(A+B)}\right) F(t) G(x-t)\end{align} is the product of three log-concave functions of $(x,t)$, so that it is log-concave.

By the Prekopa-Leindler inequality, it follows that \begin{align}\int_\R h(x,t) dt = e^{\frac{x^2}{2(A+B)}} (f*g)(x)\end{align} is log-concave as well.
\end{proof}

\section{Acknowledgements}

This research was conducted at the University of Minnesota Duluth REU program, supported by NSF/DMS grant 1062709 and NSA grant H98230-11-1-0224. I would like to thank Joe Gallian for his constant encouragement and for the wonderful environment for research at UMD. I would also like to thank Anirudha Balasubramanian for greatly helpful discussions about the Irwin-Hall distribution.

\nocite{*}

\bibliography{unimodalityandaconjectureofstanleyzanello}{}
\bibliographystyle{plain}

\ \\

\end{document}